\documentclass[11pt]{amsart}
\usepackage{color}
\usepackage{amssymb}
\theoremstyle{plain}
\newtheorem{Thm}{Theorem}

\newtheorem{Pro}[Thm]{Proposition}
\newtheorem{Lem}[Thm]{Lemma}

\newcommand{\di}{\operatorname{div}}

\newcommand{\re}{\mathbb{R}}
\newcommand{\ho}{{\mathop
{h}\limits^ \circ}}
\newcommand{\Ao}{{\mathop
{A}\limits^ \circ}}

\begin{document}

\title[Translating solitons]
{Bernstein theorem for translating solitons of hypersurfaces }

\author{Li Ma, Vicente Miquel}

\address{ Dr.Li Ma, Distinguished professor\\
Department of mathematics \\
Henan Normal university \\
Xinxiang, 453007 \\
China
}
\email{lma@tsinghua.edu.cn}

\address{ Vicente Miquel \\
Department of Geometry and Topology \\
University of Valencia \\
46100-Burjassot (Valencia), Spain
}

\email{miquel@uv.es}

\thanks{The research is partially supported by the National Natural Science
Foundation of China No.11271111 and SRFDP 20090002110019 and by the project
DGI (Spain) and FEDER  project MTM2010-15444, MTM2013-46961-P. and the G. V. Project  PROMETEOII/2014/064.. \\
 This work was done when the first named
author was visiting Valencia University in April 2014 and he would
like to thank the hospitality of the Department of geometry and
Topology.
}

\begin{abstract}
In this paper, we prove a monotonicity formula and some Bernstein type results
for translating solitons of hypersurfaces in $\re^{n+1}$, giving some conditions under which a translating soliton is a hyperplane. We also show a gap theorem for the translating soliton of hypersurfaces in
$R^{n+k}$, namely, if the $L^n$ norm of the second fundamental form
of the soliton is small enough, then it is a hyperplane.


\end{abstract}

 \maketitle

\section{introduction}

We study the translating solitons of properly immersed hypersurfaces
$F=F(x,t)\subset \re^{n+1}$, $x\in M^n, 0\leq t<T$, evolving under the mean
curvature flow defined by
$$(\partial_tF)^{\perp}=\vec{H}(F),$$
where $\vec{H}(F)$ is the mean curvature vector of the hypersurface
$F=F(x,t)$ at time $t$ and $M\subset \re^{n+1}$ is a fixed hypersurface
. These solitons are characterized by the soliton
equation
\begin{equation}\label{soliton}
H=<\nu,\omega>
\end{equation}
where $\nu$ is an outer unit normal to the fixed hypersurface $M\subset
\re^{n+1}$, $\vec{H}(F)=-H\nu$, and $\omega$ is a fixed unit vector in
$\re^{n+1}$.  When there is no confusion, we identify the position vector $x$ in $M^n$ with $F(x)$. Here and below, we use the notations as in \cite{H} such that $\vec{H}=-H\nu$ is the mean curvature vector and $A=\{h_{ij}\}:=h$ is the second fundamental form with $h_{ij}=<D_{e_i}\nu,e_j>$ and $H=h_{jj}$ for the moving orthonormal frame ${e_i}$ on $M$, where $D$ denotes the usual directional derivative in $\re^{n+1}$.  In this case the flow is given by
$$
F(x,t):=F(x)-t\omega, \ \ or \ \
\partial_tF=F_{*}(\partial_t)=-\omega,
$$
with the right side $F:M\to \re^{n+1}$ being a fixed hypersurface in
$\re^{n+1}$.

As type II singularity models of mean curvature flow, the
properties of translating solitons may be of importance to study. In
particular, the Bernstein type theorems of translating solitons are
important.

There are relatively few results about the translating
solitons. Let us just mention a few. In \cite{W}, X.-J.Wang studies
 symmetric properties of the convex graphic solitons; in \cite{Ma},
L. Ma studies the stability of the Grim Reaper, which is a
translating soliton to the curve shortening flow in the plane; in the series of papers \cite{N1,N2,N3} X. H. Nguyen constructs new examples of translating solitons; in a very recent paper \cite{MSS}, F. Martin, A. Savas-Halilaj and K. Smoczyk give, among others, some rigidity theorems for the hyperplanes and the Grim-Reaper planes and topological obstructions to their existence. For related references, we may refer to \cite{I} and \cite{Sm}.

Fix $z\in M$. Let $ S_z: \re^{n+1} \longrightarrow \re $ be the function defined by $S_z(x)=<x-z,\omega>$. WE may write by $S=S_z$ when there is no confusion and we may consider $z=0$ as the origin point. Then the soliton equation \eqref{soliton} can
be written as
\begin{equation}\label{HDS}
H = D_\nu S  \ \ \text{in}  \ \  M.
\end{equation}

Let $g$ be the metric induced on $M$ by the standard euclidean metric $<\ ,\ >$ on $\re^{n+1}$, and $dv_g$ the volume form induced by the
metric $g$ on $M$. By $\nabla$ we shall denote the Levi-Civita connection induced on $M$ by its metric $g$, and also the gradient and the differential of a function $f:M\longrightarrow \re$.

With the exception of section \ref{CGR}, from now on  we will suppose that $(M,g)$ is connected and complete.

Our first contributions about the properties of translating solitons are the following two easy observations:

\begin{Pro}\label{valencia}
 When $(M,g)$ is a translating soliton in  $\re^{n+1}$, we always have $\int_M |\nabla S(x)|dv_g=\infty$.
\end{Pro}

\begin{Pro}\label{valencia-422}
Assume that $n=2$, $(M^2,g)$ is a simply connected translating soliton which is conformal to $R^2$. Then
$\inf_MS=-\infty$.
\end{Pro}

These propositions are our motivation to look for other conditions giving
 Bernstein type theorems.
For most of those theorems we shall use on $M$ the measure
$$d\mu=e^{-S(x)}dv_g.$$
With this measure, the same ideas used to prove the propositions above will give the following result.

\begin{Pro}\label{valencia2}
Assume that $H\geq 0$ and $|\nabla H|\in L^1(M,d\mu)$, where $(M,g)$
is a  translating soliton in  $\re^{n+1}$. Then $M = \re\times \Sigma$, where $\Sigma$ is a minimal
hypersurface in $\re^{n}$.
\end{Pro}

Generally speaking the condition that $|\nabla H|\in L^1(M,d\mu)$ is
very restrictive and one may try to find other conditions weaker but still related to the second fundamental form of the translating
solitons.

\begin{Thm}\label{mono}
 When $(M,g)$ is a translating soliton in  $\re^{n+k}$, we have the following monotonicity for $0<s<t$,
 $$
 e^{-t}\mu(B_t\bigcap M)\geq e^{-s}\mu(B_s\bigcap M).
 $$
 In particular, $\mu(M)=\infty$ and $vol(B_R\bigcap M)\geq cR$ for some uniform constant $c>0$.
\end{Thm}
One consequence of this result is below.

\begin{Thm}\label{mali-VL} Assume that $M^n\subset \re^{n+1}$ is a translating
  soliton
  such that $M$ is the graph of a function $u=u(x_1,...,x_n)$,  and
$|u|\leq C$ for some uniform constant $C>0$. Then the coordinate $x_{n+1}$ can not in the direction of $\pm\omega$.
\end{Thm}

Using
ideas from the paper \cite{SSY}, we shall use these facts to
prove the following result.

We now propose a key condition to a Bernstein type result. The condition is
  \begin{equation}\label{burjasso}
    |\nabla A|\leq \frac{3n+1}{2n}|\nabla H|, \ \ in \ \ M.
  \end{equation}
One the translating soliton $M$, we have $\nabla H=<\nabla \nu, \omega>=A(\cdot, \omega)$ and the condition becomes
$|\nabla A|\leq \frac{3n+1}{2n}|A(\cdot, \omega)|$.

We shall comment on condition \eqref{burjasso} in section \ref{pkato}. Here, we just note that  every $x\in M$, there is an orthonormal basis $\{e_i\}$ of
   $T_xM$ formed by principal vectors $e_i$ of $M$ at $x$ such that in the frame $(e_i)$, $\nabla A=(h_{ijk})$, and by the Cauchy-Schwartz inequality, the condition \eqref{burjasso} implies that
    \begin{equation}\label{burja}
    \sum_{j} \nabla_jh_{jj} \nabla_jH\leq |\nabla A||\nabla H|\leq\frac{3n+1}{2n}|\nabla H|^2,
     \end{equation}
     which is the condition we shall use in section \ref{pkato}. Another remark is that we may replace the condition (\ref{burjasso}) by the condition
   $$
   |\nabla \Ao|\leq \frac{3n-1}{2n}|\nabla H|
   $$
   for the traceless part $\Ao = A-\frac1n H\ I$ of $A$.

The next is our main result stated as follows.
  \begin{Thm}\label{mali-VLC} Assume that $M^n\subset \re^{n+1}$ is a mean convex translating
  soliton
which satisfies \eqref{burjasso}. Assume further
that
\begin{equation}\label{A0L1} \int_M|\Ao |^2d\mu<\infty. \text{ (where $\Ao $ is the traceless part of $A$)}
\end{equation}
Then $M$ is a hyperplane.

\end{Thm}

Finally we state a theorem which is true also in higher codimension. In this case the equation \eqref{soliton} for the mean curvature vector of a soliton becomes
\begin{equation}\label{solitonhc}
\vec{H} = \omega^\bot
\end{equation}
where $\omega^\bot$ is the projection on the normal bundle of the submanifold $M$ of a unit vector $\omega\in\re^{n+k}$. In this case, instead of the scalar second fundamental form $A$, we shall use the vectorial second fundamental form $\alpha$ defined by $\alpha(X,Y) =(D_XY)^\bot$.

\begin{Thm}\label{gap}
Let $M^n\to R^{n+k}$ be a translating soliton, $n\geq 2$.
There exists a constant $\epsilon_0>0$ such that if
$$
\int_M |\alpha|^n dv_g\leq \epsilon_0,
$$
then $M^n$ is an $n$-dimensional plane.
\end{Thm}

Here is the plan of the
paper.  In section \ref{sect2}, we propose some elementary properties of the
translating solitons and prove propositions 1 to 3.  Sections 3 and 4 are dedicated to prove technical lemmas that will be used to prove Theorem \ref{mali-VLC} in section 5. Section \ref{CGR} is dedicated to the proof of Theorem \ref{gap} and related results on the compactness of the space of translating solitons with bounded total curvature and a decay result.

\section{Elementary properties of translating solitons}\label{sect2}

From \eqref{soliton} it is easy to compute the Hessian of $S$ in $M$. If $X,Y$ are tangent vector fields on $M$ such that $\nabla_XY(x)=0$ at $x$, then
\begin{equation}\label{HesS}
\nabla^2 S(X,Y)(x)=<D_{X}Y,\omega>= -A(X,Y) <\nu,\omega>= -H\ A(X,Y).
\end{equation}

Denote by $\Delta$ the Laplacian operator on $M$,
\begin{equation}\label{SH}
\Delta S(x)=- H<\nu,\omega>=-H^2.
\end{equation}
From the definitions of $S$ and $\nabla$, and \eqref{HDS}, follows that
$$D_XS=<\omega,X>$$ for every vector $X$ in $\re^{n+1}$
and
\begin{equation}\label{DS1}
|\nabla S|^2=|\omega^T|^2, \ \ \ |DS|^2 = H^2+|\nabla
S|^2=|\omega|^2=1.
\end{equation}
 Then we
know that $M$ is non-compact. For otherwise, assuming $M$ is
compact, we know that $S(x)$
attains its minimum at some point $x_0\in M$, where  $\nabla S(x)=0$, which implies that $|H(x_0)|=1$, and $-H^2(x_0)=\Delta S(x_0)\geq 0$ and then we have $H(x_0)=0$, a contradiction. Hence, \emph{ $M$ is non-compact}.

From \eqref{SH} one immediately  gets  the conclusion of \textbf{ Proposition \ref{valencia}}. In fact, if $|\nabla S(x)|\in
L^1(M,dv_g)$, then for suitable $R\to\infty$ , $\int_{\partial B_R\bigcap M}|\nabla S|\to 0$. Then,
$$
\int_{B_R\bigcap M} H^2dv_g=-\int_{\partial B_R\bigcap M}<\nabla
S(x), \overline{\nu}>\leq \int_{\partial B_R\bigcap M}|\nabla S|\to 0.
$$
Then $H=0$. By the well-known monotonicity formula in the minimal surface theory \cite{CM}, we know that $vol (M,g)=\infty$.  By (\ref{DS1}), we have $|\nabla S(x)|=1$, which is impossible since $|\nabla S|\in L^1$.

{\bf Remark}. One consequence of Proposition \ref{valencia} is that $vol(M,g)=\infty$. One may just use the facts $\int_M|\nabla S(x)|=\infty$ and  $|\nabla S|\leq 1$. We actually have the following more refinement. Let $B_R(p)$ be the ball in $\re^{n+1}$ with center $p\in M$ and radius $R>0$. Then there is a positive uniform constant $C_0$ such that $vol(M\cap B_R(p))\geq C_0R$ for all $R>0$. For otherwise, we have some $R_j\to \infty$ such that
$$
\frac{vol(M\cap B_{R_j}(p))}{R_j}\to 0.
$$
Let $\phi_R$ be the cut-off function in $\re^{n+1}$ such that $0\leq \phi_R\leq 1$, $\phi_R=1$ on $B_R(p)$ and $\phi_R=0$ outside $B_{2R}(p)$, and $|D\phi_R|\leq 4/R$ in $\re^{n+1}$. By (\ref{SH}), we have
$$
\int_M H^2\phi_{R_j}=\int_M <\nabla S, \nabla \phi_{R_j}>\leq 4\frac{vol(M\cap B_{R_j}(p))}{R_j}\to 0
$$
as $R_j\to \infty$, which implies that $H=0$ on $M$, i.e., $M$ is a minimal submanifold. However, on the minimal submanifold $M$, we always have the volume growth $\frac{vol(M\cap B_{R}(p))}{R^n}\geq C_1>0$ for some positive uniform constant $C_1>0$. This completes the contradictory argument.

  \textbf{Here is the proof of Proposition \ref{valencia-422}}. We argue by contradiction, i.e., $m:=\inf_M S(x)>-\infty$. By the uniformization theorem, since $M$ is simply connected, it is conformal to the euclidean plane. Then by the equation (\ref{SH}), we know that $S(x)-m\geq 0$ is non-negative superharmonic function. Since there is no nontrivial non-negative superharmonic function on $R^2$, we know that $S(x)$ is a constant function on $M$. This implies that $|\nabla S|=0$ and $H=0$ (via the equation (\ref{SH})), however, these relations imply that $H^2+|\nabla
S|^2=0$, which is impossible by (\ref{DS1}). One
immediately obtains the conclusion of  Proposition \ref{valencia-422}.

\medskip
Now, we shall give \textbf{the proof of Proposition 3}.

\begin{proof}

Recall from \cite{CM} or \cite{Simon} that
$$
\Delta h_{ij}=H_{ij}+Hh_{im}h_{mj}-|A|^2h_{ij}.
$$
Note that $H_i=<D_{e_i}\nu, \omega>=h_{ij}<e_j,\omega^T>$. By the Codazzi equation $h_{ij,k}=h_{ik,j}$ (see also \cite{SSY})
$$
H_{ik}=H_{ij,k}<e_j,\omega^T>+h_{ij}<\nabla_{e_k}e_j,\omega^T>=\nabla_{\omega^T}h_{ik}-Hh_{ij}h_{jk}.
$$
Hence we have
\begin{equation}\label{Ht}
\Delta h_{ij}=\nabla_{\omega^T}h_{ij}-|A|^2h_{ij}.
\end{equation}

The latter equation can also be obtained from the flow equation in \cite{H}
$$
\partial_t A=\Delta A+|A|^2A
$$
where $\Delta$ is the induced Laplacian operator on the hypersurface
$F(x,t)$.

 We then derive from \eqref{Ht} the
following elliptic equation for the mean curvature function $H$ on $M$:
  \begin{equation}\label{mean}
\Delta H-\nabla_{\omega^T}H + |A|^2 H=0.
  \end{equation}
  which can be written as
  \begin{equation}\label{mean1p}
  LH = - |A|^2 H, \ \text{ with } Lf = \Delta f - \nabla_\omega^T f.
  \end{equation}
   The advantage of using the operator $L$ is that
    \begin{equation}\label{mean2}
\di(e^{-S}\nabla f) = L f \ e^{-S},
  \end{equation}
  which allows us to use the divergence theorem under the form:
  \begin{align}
  &\int_\Omega L f \ d\mu = - \int_{\partial\Omega} <N, \nabla f> \ d\mu \text{ and }\nonumber \\
  &\int_M g\  L f \mu = - \int_M <\nabla f, \nabla g> \ d\mu \text { for  }g\in C_o(M).
  \label{Smu}
  \end{align}
  Since $H\geq 0$, we know from the maximum principle that either $H=0$ on $M$ or $H>0$ on $M$. We show that $H=0$ on $M$. Assume that $H>0$.
 By \eqref{mean1p} and (\ref{Smu}) one immediately gets that
$$
\int_{B_R\bigcap M} |A|^2 H d\mu = \int_{\partial B_R\bigcap M}<\nabla
H(x),\overline{\nu}>\leq \int_{\partial B_R\bigcap M}|\nabla H|\to 0
$$
 for suitable $R\to\infty$. Then we have $H=0$, i.e., $M$ is a minimal hypersurface. Then we have
 $$
 \nabla S=\omega^T=\omega, \ \ \Delta S=0, \ on \ \ M.
 $$
 By (\ref{HesS}) we know that $\nabla^2S=0$ on $M$. This says that $\omega$ is a parallel vector on $M$, which implies the splitting result as desired.
\end{proof}
In principle, the proof of
Proposition \ref{valencia2} is similar to that of Proposition 1.

\section{Monotonicity for translating solitons and a consequence} \label{sect3}
In this section, we mainly set-up the monotonicity formula in Theorem \ref{mono}. Here is the proof.

Let $w(x)=e^{-S(x)}$ in $M$. Then using (\ref{SH}) and (\ref{DS1}), we have
$$
\Delta w=w, \ \ in \ \ M.
$$
Let $f(R)=\int_{B_R\bigcap M} w$. Let $h(x)=|x|$. Then $\nabla h(x)=\frac{x^T}{|x|}$ and $|\nabla h(x)|=\frac{|x^T|}{|x|}$. Recall the coarea formula
$$
\int_{B_t\bigcap M}w=\int^t d\tau\int_{h=\tau} \frac{w |x|}{|x^T|}.
$$
Then
\begin{equation}\label{vol-1}
f'(t)=\int_{\partial B_t\bigcap M} \frac{w |x|}{|x^T|}\geq \int_{\partial B_t\bigcap M}w.
\end{equation}
 Note that $\nabla w=-w\nabla S$.
Integrating the equation $\Delta_Mw =w$ over $B_R$ we obtain
\begin{equation}\label{vol-2}
f(R)=\int_{B_R\bigcap M}\Delta w=-\int_{\partial B_R\bigcap M} w<\nabla S, \bar{\nu}>\leq \int_{\partial B_t\bigcap M}w,
\end{equation}
where $\bar{\nu}=\frac{x^T}{|x^T|}$.
Then we have
  $$
 f'(t)\geq f(t),
 $$
 which implies the monotonicity formula that for any $t>s>0$,
 \begin{equation}\label{vol-3}
 e^{-t}f(t)\geq e^{-s}f(s).
 \end{equation}
 Then $f(t)\geq e^te^{-s}f(s):=c(s)e^t\to \infty$ as $t\to\infty$. Hence $\mu(M)=\infty$.
 Note that on $\partial B_t\bigcap M$, $w\leq e^R$ and by (\ref{vol-1}) we have
 $$
 f(t)\leq f'(t)\leq e^t Area(\partial B_t\bigcap M).
 $$
 Fixing some $s=s_0$, we know from above that
 $$
 Area(\partial B_t\bigcap M)\geq c(s_0).
 $$
Then we have $vol(B_t\bigcap M)\geq c(s_0)t$.
This completes \textbf{the proof of Theorem \ref{mono}}.

 We also note that
$$
(t^{-n}f(t))'=-nt^{-n-1} f(t)+t^{-n}f'(t).
$$
Then by (\ref{vol-1}) and (\ref{vol-2}) we get
$$
-nt^{-n-1} f(t)+t^{-n}f'(t)=t^{-n-1}[\int_{\partial B_t\bigcap M} \frac{w |x|^2}{|x^T|}+n
\int_{\partial B_t\bigcap M} w<\omega, \frac{x^T}{|x^T|}>]
$$
which is not strong enough to yield the monotonicity about $t^{-n}f(t)$.

\begin{proof} (\textbf{of Theorem \ref{mali-VL}}).
It is a consequence of Theorem \ref{mono} and the following lemma..

\begin{Lem}\label{vicente} Assume that the coordinate $x_{n+1}$ is in the direction of $\omega$ and $M$ is a $n$-dimensional translating soliton is given as the graph of  the graph of a function $u(x_1,...,x_n)$, and $|u| \le C$ for some constant $C$. Then
$$
\mu(M\bigcap B_R)\leq C_n e^C R^n
$$
where $C_n$ is the euclidean volume of the ball of radius $1$ in $\re^n$.
\end{Lem}
\begin{proof} Since we are considering on $M$ the measure $\mu$ defined by the measure element $d\mu=e^{-S} dv_g$, it is convenient to use on $\re^{n+1}$ the measure $\overline\mu$ defined by the measure element $d\overline \mu = e^{-S} dv_e$, where $dv_e$ is the standard euclidean volume element. Both  measures are related, as their corresponding volume elements, by  $d\mu = \iota_\nu d\overline \mu$. Since $M$ is the graph of a function $u$, we can extend $d\mu$ over all $\re^{n+1}$ by $ d\mu_{(x_1,...,x_n,x_{n+1})} = d\mu_{(x_1,...,x_n,u(x_1,...,x_{n}))}$. This has as a consequence that
\begin{align*}
d(d\mu)_{(x_1,...,x_n,x_{n+1})} &= d(d\mu)_{(x_1,...,x_n,u(x_1,...,x_{n}))} \\
&= d(\iota_\nu d\overline\mu) = \di_S\nu\ \overline\mu= (H+<\nu,\omega>) \overline \mu = 2 <\nu,\omega> \overline \mu.
\end{align*}
Given $R>0$ and a ball $B_R(p)$, let us choose the coordinate system centered at $p$. Let $D_R$ be the ball of radius $R$ in $R^n$. Let  $V$ be the domain in $\re^{n+1}$  bounded by $M$, $\partial D_R\times R$ and $D_R\times \{-C\}$.
The application of the Stokes theorem to the integration of the extended form $d\mu$  over the boundary $\partial V$ of $V$ gives
$$
\mu(M\bigcap D_R\times R)=\int_{M\bigcap \partial V} d\mu =-\int_{\partial V-M} d\mu +\int_V 2 <\nu,\omega> \overline \mu.
$$
Note that $V$ is contained in $D_R\times [-C,C]$,  $\partial V-M \subset D_R\times\{-C\} \bigcup \partial D_r\times [-C,C]$, and $|d\mu|$ restricted to $D_R\times\{-C\} \bigcup\partial D_r\times [-C,C]$ is lower or equal than $e^C |dv_g|$, where $dv_g$ is the standard volume element on  over that hypersurface.  Then we have that the right side of above equality can be bounded by
\begin{align*}
& e^C\left(vol(\partial D_R\times [-C,C])+vol(D_R\times
\{-C\})+vol(D_R\times [-C,C])\right)\\
& \qquad \qquad \leq e^C \left( 2\omega_{n-1}R^{n-1}C+\omega_nR^n+4\omega_nR^n C\right) \\
& \qquad \qquad = e^C (2 \omega_{n-1} C + \omega_n(1+4C) R) R^{n-1}.
\end{align*}
This completes the proof of Lemma \ref{vicente}. \end{proof}

Combining this result with Theorem \ref{mono}, we have Theorem \ref{mali-VL}.
\end{proof}

\section{Kato type inequality and its consequence}\label{pkato}

In this section we present the proof of Kato inequality for hypersurfaces with the condition \eqref{burjasso}.
 This follows by our attempt considering translating solitons following the ideas in \cite{SSY} for minimal surfaces. In that paper, a basic tool in the study of minimal hypersurfaces is Kato's inequality for the second fundamental form $
|\nabla A|^2 \geq (1+\frac{2}{n})|\nabla  |A||^2.
$. In this section we shall obtain a similar inequality for the traceless part $\Ao = A-\frac{H}{n}\ I$ of $A$, under the hypothesis \eqref{burjasso} (see also \cite{B}).

\begin{Lem} Let $M^n$ be a hypersurface of $\re^{n+1}$ satisfying \eqref{burjasso}, then \begin{equation}\label{kato}
|\nabla \Ao|^2 \geq (1+\frac{1}{n})|\nabla  |\Ao||^2.
\end{equation}
\end{Lem}
\begin{proof}
We shall
identify $M$ and $F(M)$. For any $p\in M$, we choose a local orthonormal frame
$\{e_1,..., e_n\}$ such that, at $p$,
$\nabla_{e_i}e_j (p)=0$ and $\{e_1(p), ... , e_n(p)\}$ are eigenvectors of the shape operator $A$, that is, $A e_i = k_i e_i$. Let us remark that this implies that $\Ao e_i = \displaystyle \left(k_i - \frac{H}{n}\right) e_i$, that is, $\{e_i\}$ are also eigenvectors of $\Ao$.

We can choose the $\{e_i\}$ such that, at $p$, coincide with those in the hypothesis \eqref{burjasso}.

At $p\in M$, $\nabla\ho$ is in the vector space of $3$-covariant tensors $T$ on $T_pM$ satisfying the Codazzi condition $\sum_i T(X,e_i,e_i)=0$ for every $X\in T_pM$. If we consider on this space the natural metric induced by $<,>$, we can choose an orthonormal basis of it formed by the tensors
$$ \theta^i \otimes \frac{\ho}{|\ho|}, \ \theta^i\otimes S^{jk}, \ \theta^i\otimes T^{\ell} , \ 1\le i\le n, 1\le j<k\le n, 2\le \ell \le n-1,
$$
where $\{\theta^i\}$ is the dual basis of $\{e_i\}$, $S^{jk}$ are the $2$-covariant symmetric tensors whose matrix in the basis $\{e_i\}$ is $ \displaystyle S^{jk} = \left(S^{jk}_{rs}\right)_{1\le r,s \le n} = \frac1{\sqrt{2}}\left(\delta_r^j \delta_s^k + \delta_s^j \delta_r^k\right)_{1\le r,s \le n}$, and $T^\ell$ are other $2$-covariant tensors necessary to complete the basis. Using this basis we can write
\begin{align} \label{norm}
|\nabla\Ao|^2 &= |\nabla\ho|^2 \\ &= \sum_{i} <\nabla_i\ho, \frac{\ho}{|\ho|}>^2 + \sum_{i,j<k} \frac12 (\nabla_i\ho_{jk})^2+ \sum_{i,\ell} <\nabla_i\ho, T^\ell>^2\nonumber
\end{align}
But we have the following expressions for the first and second summands in the last term of the above equalities
\begin{equation}\label{norm1}
\sum_{i} <\nabla_i\ho, \frac{\ho}{|\ho|}>^2 = |\nabla |\ho||^2 = |\nabla |\Ao||^2
\end{equation}
\begin{align}\label{norm2}
\sum_{i,j<k} \frac12 (\nabla_i\ho_{jk})^2 &= \sum_{i,j\ne k}  (\nabla_i\ho_{jk})^2 \ge \sum_{j\ne k}  (\nabla_j\ho_{jk})^2 \nonumber \\
& = \sum_{j\ne k}  (\nabla_jh_{jk} -\frac1n \nabla_jH \delta_{jk})^2 = \sum_{j\ne k}  (\nabla_kh_{jj})^2  \nonumber \\
&= \sum_{j\ne k}  (\nabla_k\ho_{jj} + \frac{1}{n}\nabla_kH g_{jj})^2 \nonumber \\
&=\sum_{j\ne k}  (\nabla_k\ho_{jj})^2 + \frac{n-1}{n^2}|\nabla H|^2 + \frac2n \sum_{j\ne k}\nabla_k H \nabla_k\ho_{jj}   \nonumber \\
&=\sum_{j\ne k}  (\nabla_k\ho_{jj})^2 + \frac{n-1}{n^2}|\nabla H|^2 + \frac2n |\nabla H|^2 - \frac2n\sum_{j}\nabla_j H \nabla_j\ho_{jj}
\end{align}
where we have used the Codazzi equation for the third equality. Now, let us observe that \eqref{burjasso} implies that
$$\sum_{j}\nabla_j H \nabla_j\ho_{jj} \le \frac{3n-1}{2 n}|\nabla H|^2, $$
then, by substitution of this inequality in \eqref{norm2}, we obtain
\begin{align}\label{norm3}
\sum_{i,j<k} \frac12 (\nabla_i\ho_{jk})^2 &= \sum_{i,j\ne k}  (\nabla_i\ho_{jk})^2 \ge \sum_{j\ne k}  (\nabla_j\ho_{jk})^2 \nonumber \\
& \ge \sum_{j\ne k}  (\nabla_k\ho_{jj})^2 .
\end{align}
Moreover, since $\ho$ is diagonal in the basis $\{e_i\}$,
\begin{align}\label{norm4}
|\nabla |\Ao||^2 &= \sum_{k} <\nabla_k\ho, \frac{\ho}{|\ho|}>^2 = \sum_{k} <\nabla_k\ho, \frac{\ho}{|\ho|}>^2 \nonumber \\
&= \frac1{|\ho|^2} \sum_{k} (\sum_j \nabla_k\ho_{jj}  \ho_{jj})^2 \nonumber \\
&\le  \frac1{|\ho|^2}\sum_{k} (\sum_j (\nabla_k\ho_{jj})^2) (\sum_j ( \ho_{jj})^2) = \sum_{k,j} (\nabla_k\ho_{jj})^2 \nonumber \\
&=\sum_{k\ne j} (\nabla_k\ho_{jj})^2 + \sum_{j} (\nabla_j\ho_{jj})^2 \nonumber \\
&=\sum_{k\ne j} (\nabla_k\ho_{jj})^2 + \sum_{j} (\sum_{i\ne j}\nabla_i\ho_{jj})^2 \nonumber \\
&\le\sum_{k\ne j} (\nabla_k\ho_{jj})^2 + \sum_{j} (n-1)  \sum_{i\ne j}(\nabla_i\ho_{jj})^2 = n \sum_{k\ne j} (\nabla_k\ho_{jj})^2
\end{align}

From \eqref{norm}, \eqref{norm3} and \eqref{norm4}, we obtain \eqref{kato}.
\end{proof}
As a corollary, we obtain

\begin{Lem}\label{simons} Let $M^n$ be a translating soliton of $\re^{n+1}$ satisfying \eqref{burjasso}, then \begin{equation}\label{katos2}
|\Ao |L |\Ao |+|A|^2|\Ao |^2\geq \frac{1}{n}|\nabla
|\Ao ||^2.
\end{equation}
\end{Lem}
\begin{proof}
From (\ref{Ht}), we have
\begin{equation*}
 \Delta |A|^2=\nabla_{\omega^T}|A|^2+2 |\nabla A|^2 - 2 |A|^4
\end{equation*}
We may also think about the flow $F(x,t) = F_0(x) - t\ \omega$ for the translating soliton, we have  $\partial_t|A|^2(x,0)= - \nabla_{\omega^T}|A|^2 (x)$ for every $x\in M$.
Hence we have
\begin{equation}\label{key}
\Delta |A|^2-\nabla_\omega |A|^2=2|\nabla A|^2-2|A|^4.
\end{equation}
Note that
$$
|\Ao|^2 = |A - \frac H n I|^2 = |A|^2 -\frac {H^2} n.
$$
Using \eqref{mean} and \eqref{key}, we obtain the following formula
\begin{equation}\label{key2}
\Delta |\Ao|^2-\nabla_{\omega^T} |\Ao|^2 = 2 |\nabla \Ao|^2 - 2|A|^2 |\Ao|^2.
\end{equation}
Then \eqref{katos2} follows from this equality and inequality \eqref{kato}.
\end{proof}

{\bf Some remarks on condition \eqref{burjasso}}.This condition appears as a necessary technical condition to have the Kato's type inequality \eqref{kato} for the tensor $\Ao$. Here we want to grasp a little this condition by looking at its meaning in two simple families of hypersurfaces.

The first are surfaces $\Gamma\times \re^{n-1}$, with $\Gamma$ a curve in $\re^2$. If we denote by $e_1$ the unit vector tangent to $\Gamma$, any orthonormal basis used to write condition \eqref{burjasso} contains $e_1$, and $H= k$, the curvature of the curve $\Gamma$ in $\re^2$. Then condition \eqref{burjasso} just states $|D_{e_1}k|^2 \le \frac{3n+1}{2 n} |D_{e_1}k|^2$, which is always true.

The second family to consider are revolution surfaces obtained by the rotation of a curve $c(s)=(x_1(s), x_{n+1}(s))$, parametrized respect its arc length $s$,  in the plane $\{x_1, x_{n+1}\}$ around the axis $X_{n+1}$. In this case we are forced to take as one the the vectors (say $e_1$) of the orthonormal basis the unit vector tangent to $c(s)$ (or those obtained by rotation of it), because this curve is a curvature line. For such a revolution surface, $h_{11} = k$, the curvature of $c$, and, for $j\ne 1$, $$\displaystyle h_{jj} = \frac1{x(s)}<(-1,0),(-x_{n+1}'(s), x_1'(s))> =\frac{x_{n+1}'(s)}{x(s)}.$$ Both curvatures have derivative $0$ in the directions $e_j$, $j\ge 2$, then a stronger condition than \eqref{burja} is that
$$\displaystyle (D_{e_1}k) D_{e_1}(k + (n-1) \frac{x_{n+1}'(s)}{x(s)}) \le \frac{n+1}{2 n}(D_{e_1}(k + (n-1) \frac{x_{n+1}'(s)}{x(s)}))^2,$$ that is
$$\displaystyle \frac{n-1}{2 n} k'(s)^2 \le \frac{n-1}{n} k'(s) \left(\frac{x_{n+1}'}{x}\right)'(s) + \frac{n+1}{2 n} \left(\frac{x_{n+1}'}{x}\right)'(s)^2.$$ Obviously this inequality is satisfied when $c$ is a line or a circle. Explicit computations with concrete functions show that it is satisfied, for instance, when $c$ is the graph of the functions $x_{n+1}(x_1)= \sqrt{x_1}$,  $x_1^2$ for $n\le 7$, $\sinh(x_1)$, $\cosh(x_1)$.


\section{proof of Theorem \ref{mali-VLC} and related} \label{sect5}

In this section we give a basic analytic lemma which will be used to give a
Bernstein type theorem.

 \begin{Lem}\label{mali8} Assume that $M^n\subset \re^{n+1}$ is a
  hypersurface such that there are positive functions $u>0$ and $B>0$ on $M$ such that
  \begin{equation}\label{simons8}
uLu+B u^2\geq c_0 |\nabla u|^2
  \end{equation}
for some uniform constant $c_0>0$ and with the \emph{stability} condition
\begin{equation}\label{stable8}
\int_M(|\nabla \phi|^2- B \phi^2)\ d\mu \geq 0
\end{equation}
for
  any $\phi\in C_0^2(M)$.
 Assume that $B \geq b\ u^2$ for some constant
$b>0$. Then, there is a uniform constant $C>0$, for
  any $\eta\in C_0^2(M)$,
  \begin{equation}\label{interm-1}
\int |\nabla u|^2\eta^2\leq C \int u^{2}|\nabla \eta|^2,
\end{equation}

  moreover, there is a small $\varepsilon >0$ such that for every
  $$\displaystyle p\in\left]4-2\sqrt{\frac{c_0}{1+\varepsilon}} , 4+2\sqrt{\frac{c_0}{1+\varepsilon}}\right[,
  $$ there is a constant $C(n,p)$  such that
\begin{equation}\label{moser}
\int_M u^{p}\eta ^{p} d\mu\leq C(n,p) \int_M (|\nabla
\eta|^{p})d\mu
\end{equation}
\end{Lem}

\begin{proof} We shall follow an idea from \cite{SSY}. In below, all the
integrals are along $M$ and  with respect to the measure $d\mu$. Moreover we shall use the divergence formulas \eqref{Smu}.

Let us take $\phi=u^{1+q}\eta$ in the stability inequality
(\ref{stable8}). We get that
\begin{align}
\int Bu^{2+2q}\eta^2\leq & \int
(1+q)^2u^{2q}|\nabla u|^2\eta^2 + \int u^{2+2q}|\nabla \eta|^2 \nonumber \\
& \quad +2(1+q)\int u^{1+2q}\eta\ <\nabla u,\nabla \eta>. \label{stable2} \\
& = \int u^{2q} |(1+q) \eta \nabla u + u \nabla \eta|^2 \label{stable3}
\end{align}
Multiplying (\ref{simons8}) by $\eta^2 u^{2q}$ and integrating over
$M$ using the divergence formula, having into account that $\eta$ has compact support, we get that
$$
\int c_0  u^{2q}\eta^2|\nabla u|^2\leq -(1+2q)\int
u^{2q}|\nabla u|^2\eta^2+ \int B u^{2+2q}\eta^2
$$
$$
-2\int  u^{1+2q} \eta <\nabla u,\nabla \eta>.
$$
Applying (\ref{stable2}) to the above inequality we have
$$
\int c_0  u^{2q}\eta^2|\nabla u|^2\leq q^2\int u^{2q}|\nabla
u|^2\eta^2+ 2q\int  u^{1+2q} \eta<\nabla u,\nabla \eta> + \int u^{2q} |q \eta \nabla u + u \nabla \eta|^2
$$

 Using the
Cauchy-Schwartz inequality and Young's inequality
$$
2\int  u^{1+2q} \eta<\nabla u,\nabla \eta>\leq\epsilon q^2\int
\eta^2 u^{2q}|\nabla u|^2+\epsilon^{-1}\int  u^{2+2q}|\nabla
\eta|^2,
$$
we get that
\begin{align}\label{key5}
[c_0-(1+\epsilon) q^2]  \int  u^{2q}|\nabla  u|^2 \eta
\leq (1+\epsilon^{-1})\int  u^{2q+2 } |\nabla\eta|^2.
\end{align}

Let us take $p=2q+4>0$ (which implies $p>4$ in order $q>0$)  in (\ref{stable2}) and $q^2<c_0$. Choose $\epsilon>0$
small such that $c_0-(1+\epsilon)q^2>0$. Then
\begin{equation}\label{interm}
\int u^{p-4}|\nabla u|^2\eta^2\leq C \int u^{p-2}|\nabla \eta|^2
\end{equation}
for some constant $C>0$. Choose $p=4$ and we get (\ref{interm-1}).

Using the fact $B\ge b u^2$,  \eqref{stable3} and the inequality $|x+y|^2 \le 2|x|^2 + 2 |y|^2$,
\begin{align*}
b \ \int u^p \eta^2 \le \int u^{p-2} B \eta^2 \le u^{p-4} |(1+q)\eta \nabla u + u \nabla \eta|^2  \\
\le \int 2 u^{p-4} \left( (1+q)^2 \eta^2 |\nabla u|^2 + u^2 |\nabla\eta|^2\right)
\end{align*}
Applying now \eqref{interm}, we obtain
\begin{equation}\label{ecu}
b \ \int u^p \eta^2\ \le\ 2\ C\ ((1+q)^2 +1) \int u^{p-2} \nabla \eta|^2
\end{equation}
Now we use the the Young's inequality
$$
 u^{p-2}|\nabla \eta|^2 = u^{p-2} \eta^{2(p-2)/p} \frac{|\nabla \eta|^2}{\eta^{2(p-2)/p}} \leq \delta  u^p \eta^2 + C_\delta  \frac{|\nabla \eta|^p}{\eta^{p-2}},
$$
and the substitution of this inequality in \eqref{ecu} gives
$$
\int  u^p\eta^2\leq C_1\int (|\nabla \eta|^p\eta^{2-p})
$$
for some constant $C_1$. Replacing $\eta$ by $\eta^{p/2}$ we then get (\ref{moser}).
\end{proof}

\begin{proof} {\it of Theorem \ref{mali-VLC}}.
 We plan to show that under the assumption $H\geq 0$, we must have
$\Ao =0$, which implies that $M$ is umbilical everywhere on $M$
and, because $M$ is complete and non compact, this tells that $M$ is  a hyperplane.

Translating solitons are critical points of the volume functional $V_S(\Omega) = \int_\omega d\mu$, $\Omega\subset M$. A computation of  the second variation formula similar to which is done for minimal surfaces shows that a translating soliton $M$ is stable for the functional $V_s$ if and only if, for any $\phi\in C_0^2(M)$,
\begin{equation}\label{stable}
\int_M(|\nabla \phi|^2-|A|^2\phi^2)\ d \mu\geq 0.
\end{equation}

  We remark that for any nontrivial mean convex soliton, that is, $H\geq 0$, by  maximum principle applied to \eqref{mean} we must have $H>0$. This then implies the stability condition \eqref{stable}, as follows   by an standard argument (for instance, see the argument in pages 46-47 of \cite{CM} for minimal submanifolds)

Then, the function $B=|A|^2$ satisfies \eqref{stable8}. Moreover, by the hypothesis \eqref{burjasso}  and Lemma \ref{simons}, $u=|\Ao|$ satisfies \eqref{katos2}, that is, satisfies \eqref{simons8} with $c_0=1/n$, and $B \ge b u^2$ with $b$=1. Then we can apply the Basic Lemma \ref{mali8} to conclude \eqref{interm-1}.

Let  $r(x)$ denote  the distance from $x$ to the origin in
$\re^{n+1}$. In \eqref{interm-1}, choose $\eta=\eta(r(x))$ to be the cut-off
function defined by
\begin{equation}\label{eta}
\eta(r)=1, \ \text{if}  \ r\leq \theta R; \   \eta(r)=0, \  \text{for} \
r\geq R, \text{ and linear for  }r\in [\theta R,R]
\end{equation}
where  $\theta\in (0,1)$ is any fixed constant.

Then we can apply again the the computations in the proof of Lemma \ref{mali8}, with the function $\eta$  defined in \eqref{eta} and $q=0$ in (\ref{key5}) to obtain
$$
\int_{B_{R/2}}|\nabla|\Ao ||^2d\mu\leq cR^{-2}
\int_M|\Ao |^2d\mu
$$
which goes to $0$ as $R\to\infty$ because of the hypothesis \eqref{A0L1}. Hence we have
$$
\nabla|\Ao |=0.
$$
Then $|\Ao|$ is constant. This and  $\mu(M)=\infty$ tell us that that $\Ao=0$. Then  $M$ is a hyperplane.
\end{proof}

\section{compactness and Gap results}\label{CGR}

The purpose of this section is to prove Theorem \ref{gap}. This result is the analog, for translating solitons, of Corollary 2.3 in \cite{A} for minimal submanifolds of $\re^n$. The proof of it and the necessary preliminary results follow the same arguments than in that paper and, some parts, in \cite{CS}. Then we'll indicate only the points where the condition of being minimal is used in \cite{A} or \cite{CS} and how things still work when $\vec{H}=0$ is changed by condition \eqref{solitonhc}.

The proof of Theorem \ref{gap} relies on the following sequence of lemmas, where $N=n+k$.

\begin{Lem}[Compactness Theorem]\label{compactness} Let $\{M_j^n\}$ be a sequence of
connected translating solitons in $B^{N}(1)$ such that $\partial
M_j^n\bigcap B^N(1)=\varnothing$. Suppose that there is an uniform
constant $C>0$ such that $\sup |\alpha_j|(x)\leq C$ for all $j$. Then
there is a subsequence of $(M_j)$, still denoted by $(M_j)$ that
converges in the $C_{loc}^\infty$ norm to a smooth translating
soliton $M_\infty$ in $B^N(1)$ with $\sup |\alpha_\infty|(x)\leq C$.
\end{Lem}
This Lemma is stated and proved  for minimal submanifolds (in different situations) in \cite{A}, \cite{CS} and \cite{Lan}. The proof of \cite{A} and \cite{CS} works also for translating solitons. In fact, the condition $H=0$ is used is to state that, locally, $M_j$ can be written as a graph of a function $f_j$ satisfying the elliptic equations system $\mathbb{M}(f_j)=0$, where $\mathbb M$ is the operator giving the mean curvature of the graph of $f_j$. For translating solitons we only need to change this equation by $\mathbb{M}(f_j)= \omega_j^\bot$, which is still an elliptic system, with $|\omega^\bot| \le 1$,  and the rest of the arguments is like in the quoted papers.

We can get the following two assertions. The first one is the Heinz
type estimate.

\begin{Lem}\label{ass1} There is a constant $\epsilon_0>0$ such that if $F:M\to
R^{N}$ is a translating soliton with $D_1(p)\cap
\partial M=\varnothing$ for some $p\in M$ and with
$$
\int_{D_1(p)}|\alpha|^n dv_g\leq \epsilon_0,
$$
then
\begin{equation}\label{A1}
\sup_{s\in [0,1]}[s^2 \sup_{D_{1-s}(p)}|\alpha|^2]\leq 4.
\end{equation}
\end{Lem}
The proof follows exactly the arguments of the proof  Assertion (S') in \cite{A}. Minimality condition in that argument is used to apply compactness theorems of a sequence of immersions
 $\widetilde F_i:M_i\to R^{N}$ which are rescalings $\widetilde F_i = |\alpha|_i(y_i)F_i$ of minimal immersions $F_i$. In this case, if the $F_i$ are translating solitons, the mean curvature $\vec{\tilde H_i}$ of the rescaled immersion will satisfy $\vec{\tilde H_i} = w_i:= |\alpha|_i(y_i)^{-1} \omega_i^\bot$. If $|\alpha|_i(y_i)\ge 1$, then $|w_i|\le 1$, and the compactness theorem \ref{ass1} is still true on this family. If $|\alpha|_i(y_i)\le 1$ for some $i$, we can still have a family satisfying the hypothesis of the compactness theorem taking $\widetilde F_i = F_i$ for this $i$. With this small change in the definition of the $\widetilde F_i$, the argument in \cite {A} works also here.

 \begin{Lem}\label{ass2} For any small constant $\epsilon_0>0$ there is a constant $\delta>0$ such that if $F:M\to
\re^{N}$ is a translating soliton with $D_1(p)\cap
\partial M=\varnothing$ for some $p\in M$ and with
$$
\int_{D_1(p)}|\alpha|^n dv_g\leq \epsilon_0,
$$
then
\begin{equation}\label{A2}
\sup_{D_{1/2}(p)}|\alpha|^2\leq \delta.
\end{equation}
\end{Lem}

This Lemmas can be proved from the above one like in \cite{A} with no change. Also the proof of Theorem \ref{gap} follows from Lemma \ref{ass2} like in \cite{A} with no change.

Once we know these lemmas, the following results are proved following exactly the same arguments that in \cite{A}

\begin{Pro}\label{decay}
Let $M^n\to \re^{N}$ be a complete translating soliton, $n\geq 2$
with $\int_M |\alpha|^n<\infty$. Fix
$0\in M$. Then there is a uniform constant $R_0>0$ such that
\begin{equation}\label{total}
\sup_{x\in \partial B_{R}(0)} |\alpha|^2(x)\leq R^{-2} \lambda
(\int_{B(R/2,2R)} |\alpha|^n dv_g)
\end{equation}
for all $R\geq R_0$, where $\lambda(\epsilon)\to 0$ as $\epsilon\to 0$
and $B(R/2,2R)=B_{2R}(0)-B_{R/2}(0)$.
\end{Pro}

We remark that using the Sobolev inequality \cite{MS} we have the
following gap result.

\begin{Thm} \label{mali2} Let $M^n\to \re^{N}$ be a complete translating soliton, $n\geq 2$. There is a small positive constant $\epsilon(n)$ such
that if for some $p\in (1,n)$,
$$
\int_M |\alpha|^p dv_g\leq \epsilon(n),
$$
then $\alpha=0$, that is, $M$ is a $n$-plane.
\end{Thm}

The proof of above result is similar to that of Theorem \ref{gap}. We only
give an outline of the proof. In fact, one can use Moser's iteration
argument and the Sobolev inequality \cite{MS} to conclude that $|\alpha|$
is uniformly bounded, saying $|\alpha|\leq C$ for some uniform constant
$C>0$. Then we have
$$
\int_M |A|^ndv\leq C^{n-p}\int_M |A|^pdv\leq C\epsilon(n).
$$
Then we can apply Theorem \ref{gap} to get the conclusion of Theorem
\ref{mali2}.

\end{document}